\documentclass[11pt]{article}
\usepackage{enumerate}
\usepackage{amsmath}
\usepackage{amssymb,latexsym}
\usepackage{amsthm}
\usepackage{color}
\usepackage[normalem]{ulem}
\usepackage{graphicx}
\usepackage{hyperref}
\usepackage{amscd}
\usepackage{todonotes}

\usepackage{comment}
\usepackage{adjustbox}

\newcommand{\Z}{{\mathbb Z}}
\newcommand{\QQ}{{\mathbb Q}}

\newcommand{\cC}{{\mathcal C}}

\newcommand{\cH}{{\mathcal H}}

\newcommand{\Q}{{\mathbb Q^*}}

\newcommand{\s}{{\mathrm S}}
\newcommand{\h}{{\mathrm h}}

\newtheorem{theorem}{Theorem}[section]
\newtheorem{problem}[theorem]{Problem}
\newtheorem{lemma}[theorem]{Lemma}

\newtheorem{definition}[theorem]{Definition}

\newtheorem{nota}[theorem]{ Notation}
\newtheorem{constr}[theorem]{\bf Construction}
\newtheorem{cor}[theorem]{\bf Corollary}
\newtheorem{obs}[theorem]{\bf Observation}

\newtheorem{prop}[theorem]{\bf Proposition}
\newtheorem{conj}[theorem]{\bf Conjecture}

\theoremstyle{definition} 
\newtheorem{defi}[theorem]{\bf Definition}

\title{Maximizing the number of rational-value sums or zero-sums}
\author{Benjamin Móricz \footnote{ ELTE Eötvös Loránd University, Budapest, Hungary. E-mail: {\tt moriczben@student.elte.hu}} \quad Zoltán Lóránt Nagy \footnote{ELTE Linear Hypergraphs  Research Group, ELTE Eötvös Loránd University, Budapest, Hungary.  The author is supported by the János Bolyai Scholarship of the Hungarian Academy of Sciences, and the University Excellence Fund of Eötvös Loránd University. E-mail: {\tt nagyzoli@cs.elte.hu}}}
\date{}%February 2024}

\begin{document}
	
	\maketitle
	
	\begin{abstract}
		What is the maximum number of $r$-term sums admitting rational values in  $n$-element sets of irrational numbers? We determine the maximum when $r<4$ or $r\geq n/2$ and also in case when we drop the condition on the number of summands. It turns out that the $r$-term sum problem is equivalent to determine  the maximum number of $r$-term zero-sum subsequences in $n$-element sequences of  integers, which can be seen as a variant of the famous Erdős-Ginzburg-Ziv theorem.\\
	\textit{	Keywords:} zero-sum theorems, Erdős-Ginzburg-Ziv, sumset sums, Sperner-theorem, irrational numbers
	\end{abstract}
	
	\section{Introduction}
	
	Extremal graph theoretic problems usually seek the maximum (minimum) number of edges or hyperedges of an $n$-vertex  graph or hypergraphs which fulfil certain constraints, such as not containing certain subgraphs in the case of Turán type problems. 
	We address such an extremal graph theoretic problem with an underlying algebraic structure: what is the maximum number of $k$-term sums admitting rational values in an $n$-element sets of irrational numbers. 
	
	More formally, we propose the following problem. Let $\Q:=\mathbb R\setminus \QQ$ denote the set of irrational numbers.
	
	\begin{problem} \label{elsoproblema}
		Let $A$ denote a set of $n$ irrational numbers and $r>1$ a positive integer.  Let $\cH(A,r)$ denote the set of $r$-element subsets $A'\subseteq A$ for which $\sum_{x\in A'} x\in \mathbb{Q}$. Determine $$\h(n,r):=\max_{A\in \Q^n} |\cH(A,r)|.$$
	\end{problem}
	
	Several extremal problems have been investigated before where the underlying set is a set of numbers and (hyper)edges are corresponding to certain sums. We mention here the celebrated Manickam-Miklós-Singhi theorem \cite{Miklos, Singhi}, and the recent progress on the related conjecture \cite{stability_Alon, Alexey}.  Recall that this conjecture states that for positive integers $n, k$,  every set of $n$ real numbers with non-negative sum has at least $\binom{n-1}{k-1}$ $k$-element subsets whose sum is also non-negative, provided that  $n \geq 4k$. \\
	Another related area is zero-sum problems, see the excellent surveys of Caro \cite{Caro} and of Gao and Geroldinger \cite{Gao}. We return to the link to this area at the end of the paper, but we point out here that Problem \ref{elsoproblema} is equivalent to another very natural additive combinatorial question, formulated below. Their equivalence will be proved in Section 2.
	
	\begin{problem}\label{elsoprob_variant} Given a sequence of $n$ nonzero integer (or rational) numbers, what is the maximum number
		of $r$-term zero-sum subsequences?
	\end{problem}

	It is very natural to pose similar questions for general group setting as well, in the spirit of the well known Davenport constant and Erdős-Ginzburg-Ziv constant, see \cite{egz1, Gao, egz2} and the references therein. 
	
	\begin{problem}\label{main_genp}
		Let $G$ be an  additive  group and  $k<n$ integer numbers. What is the maximum number of $k$-term subsequences of an $n$-element sequence of nonzero elements of $G$, which has zero sum?  
	\end{problem}

	As a matter of fact, motivated by  the statement of the Erdős-Ginzburg-Ziv theorem, Bialostocki raised the following question.
	
	\begin{conj}[Bialostocki, \cite{Bia}]\label{biaconj}
		Let $A = \{a_1, a_2, \ldots , a_n\}$ be a sequence of integers.
		Then $A$ contains at least $\binom{\lfloor n/2 \rfloor}{m} +\binom{\lceil n/2 \rceil}{m}$  zero-sum subsequences $\pmod m$ of length $m$.
	\end{conj}
	
	This has been confirmed  by Kisin \cite{Kisin} in the  case  $m$ is a prime power $p^h$, or of form $p^h\cdot q$ for primes $p,q$ and by Füredi and Kleitman  when $m$ is either of form $m=pq$ or when $n$ is very large compared to $m$ \cite{Furedi}. Later Grynkiewicz proved the conjecture when $n$ is not very large compared to $m$, more precisely, in the case $n<\frac{19}{3}m$ \cite{Gryn}.\\
	Hence, Conjecture \ref{biaconj} asks for the \textit{minimum} of these, while  Problem \ref{main_genp} seeks for the \textit{maximum number of zero-sums}, complementing the former Erdős-Ginzburg-Ziv type question.
	
	\begin{nota} 
		We use the standard additive combinatorial notation for the sumset $A+B$ of two subsets $A$ and $B$ of an abelian group $G$ as the set of all sums $A + B = \{a+b : a \in A, b \in B\}$. If a set consists of a single element, we omit the brackets for brevity.
	\end{nota}
	
	Our aim is to give exact solutions to Problem \ref{elsoprob_variant} for various values of $r$. As it will turn out, this answers the former problem as well, Problem \ref{elsoproblema}, which is also motivated from extremal graph theory.
	
	The case $r=2$ of Problem \ref{elsoproblema}, corresponding to ordinary graphs with edges assigned to rational sums $a_i+a_j$, is a folklore 
	exercise. Here one can easily show that odd cycles are not contained in the graph induced by the set of irrational numbers, thus the maximum can not exceed the maximum number of edges of bipartite graphs on $n$ vertices. On the {other hand}, consider a set $A=A_1\cup A_2$ such that $A_1\subset \mathbb{Q}+\alpha$, $A_2\subset \mathbb{Q}-\alpha$, with $||A_1|-|A_2||\leq 1$. This construction implies that  $\max_A |\cH(A,2)|=\lfloor\frac{n^2}{4}\rfloor$. Note that the result can be deduced from Mantel's theorem as well.
	
	%short exposition\todo{ez egy todo note}
	%\textcolor{blue}{(note in blue)}

	Our first contribution is the complete solution for the case $r=3$.
	
	\begin{theorem}\label{r=3} 
		Let $n$ be a natural number of form $n=3k+\ell$ with  $0\leq \ell<2$.
		Then the maximum number of rational $3$-term sums in $n$-element irrational sets is $$\h(n,3)=k\cdot \binom{2k+\ell}{2}.$$
	\end{theorem}
	
	Next we present the exact result for $r\geq\frac{n}{2}$.
	\begin{theorem}\label{nagyr}
		Let $n/2<r$ hold for $n, r\in \mathbb{N}$.
		Then the maximum number of rational $r$-term sums is $\h(n,r)= \binom{n-1}{r-1}$.\\
		Suppose that  $n/2=r$ holds.
		Then $\h(n,\frac{n}{2})= 2\binom{n-2}{r-1}$.
	\end{theorem}
	
	In general, we conjecture that the lower bound below is tight.
	\begin{theorem}\label{mainconst}
		Let $n/2>r$ hold for $n, r\in \mathbb{N}$.
		Then  $\h(n,r)\geq \binom{n-\lfloor\frac{n}{r} \rfloor}{r-1}\cdot \lfloor\frac{n}{r} \rfloor$ holds.
	\end{theorem}
	
	It is also natural to study the non-uniform version of Problem \ref{elsoproblema}.
	
	\begin{problem} \label{nemuniform}
		Let $A$ denote a set of $n$ irrational numbers.  Let $\cH(A)$ denote the set of  subsets $A'\subseteq A$ for which $\sum_{x\in A'} x\in \mathbb{Q}$. Determine $\h(n):=\max_{A\in \Q^n} |\cH(A)|$.
	\end{problem}

	\begin{theorem}\label{osszes}  $\h(n)={n\choose {\lfloor \frac{n}{2} \rfloor}}$
		is the maximum number of rational sums in $n$-element sets of irrationals.    
	\end{theorem}
	
	The numerical result has a resemblance with the statement of the classical Sperner's theorem. In fact, our proof will use ideas similar to one of its proofs (see, e.g., in \cite{Patkos}), namely chain decompositions.
	
	Finally, since our ground set has a natural ordering, we also discuss the following variant.
	
	\begin{problem} \label{ordered}
		Let $A$ denote a set of $n$ irrational numbers $a_1< a_2< \ldots< a_n$.  Let $\cH_{ord}(A)$ denote the set of subsets $A'$ of form $A'=A\cap I$ for which $\sum_{x\in A'} x\in \mathbb{Q}$ and $I\subset \mathbb{R}$ is an interval. Determine $\h_{ord}(n):=\max_{A\in \Q^n} |\cH_{ord}(A)|$.
	\end{problem}

	%A bizonyítás a Sperner tétel láncfelbontási bizonyításához hasonlóan megy, c.f. \cite{Patkos}
	\begin{theorem}
		$\h_{ord}(n)=\lfloor {\frac{n^{2}}{4}} \rfloor.$
	\end{theorem}
	
	Our paper is organized as follows. In Section 2 we introduce a key idea which helps to tackle the case $r=3$ and to reveal the connection to Problem \ref{elsoprob_variant} as well. We then complete the proof of Theorem \ref{r=3} and Theorem \ref{nagyr}. Section 3 is devoted to the proof Theorem \ref{nemuniform} followed by the proof Theorem \ref{ordered} in Section 4. Finally, pose some open problems in Section 5.

	\section{Irrational sum free sets - uniform setting}
	
	We begin with a key lemma that will be applied several times.
	
	\begin{lemma}[Key lemma on homomorphisms]\label{key_homo}
		Let $A$ denote a set of irrational numbers. Consider the vector space generated by $A$ over  $\mathbb{Q}$, $\langle A\rangle\subset \mathbb{R}$. As it has a finite basis over $\mathbb{Q}$, there exists a vector space homomorphism $\phi : \langle A\cup \{1\} \rangle \to \mathbb{R}$, such that $\ker(\phi)=\mathbb{Q}$. 
	\end{lemma}
	Such a homomorphism will be referred to as a \textit{core homomorphism} with respect to $A$. This lemma enables us to focus on the maximum number of zero-sums instead of rational ones.
	%\begin{proof}
	%\end{proof}

	\begin{lemma} $\h(n,r)$  can be attained on a set $A\in \Q^n$ which is contained in a $1$-dimensional subspace of a $2$-dimensional vector space $\QQ(\alpha)$ for some $\alpha\in \Q$, more precisely, $A\subset \alpha\QQ$.    
	\end{lemma}
	\begin{proof}
		Suppose that $\h(n,r)=\h(n,A)$, and $\langle A \rangle_{\QQ}$ has dimension $d$ over $\QQ$. We may also assume that $\QQ\cap \langle A \rangle_{\QQ}=\{0\}$ in view of Lemma \ref{key_homo}, since otherwise the application of a core homomorphism w.r.t. $A$ would provide such a set.
		Then we can pick $d$ elements $\{a_1, a_2, \ldots, a_d\}$ from $A$ which form a basis in $\langle A \rangle_{\QQ}$, hence every $a_i$ for $ i\in \{1, 2, \ldots, n\}$ is a linear combination  of the basis (over $\QQ$). Apply the linear map $\psi_{\lambda}$ defined by 
		
		$\psi_{\lambda}:\begin{cases}
			a_1 \to a_1\\
			a_2 \to  \lambda\cdot a_1\\
			a_i\to a_i & \text{ for every } i\in \{3, \ldots d\},
		\end{cases}$
		
		for $\lambda\in \QQ$. Apart from finitely many $\lambda\in \QQ$, this 
		maps the elements of $A$ to $\Q$, hence there is a $\lambda$ for which $\psi(A)\in \Q^n$ and $\h(n,A)\le\h(n,\psi(A))$ such that the dimension $d$ of the vector space generated by the set of irrationals is decreased by one, provided that $d>1$. The lemma thus follows.
	\end{proof}
	By the application of this lemma, we get Problem \ref{elsoprob_variant} as an equivalent reformulation to our main problem.
	
	%\begin{problem}\label{zerosum} Given a sequence of $n$ nonzero rational numbers, what is the maximum number of $r$-term  zero-sum subsequences?\todo{ezt az introba is kiemelni, sőt esetleg az abstraktba}
	%\end{problem}

	\subsection{$3$-term sums, proof of  Theorem \ref{r=3}}
	Now we are ready to prove Theorem \ref{r=3}.
	First we apply Lemma \ref{key_homo} and thus we assume that every rational sum is zero. Then let $p$ denote the number of positive elements in $A$. Without loss of generality, we may also assume that $p\ge n-p.$\\ It is clear that  it is possible to get $\binom{p}{2}(n-p)$ $3$-term zero-sums if we have $p$ elements equal to $\alpha$ and $n-p$ elements equal to $-2\alpha$ for some $0< \alpha\in \Q$.\\
	On the other hand, observe that every zero-sum contains at least one positive and a negative number. A pair of a positive and a negative number uniquely determines a positive or negative number with which they form a zero-sum triple. Thus the number of zero-sum triples is at most $p\cdot(n-p)\cdot\max\{p-1, n-p-1\}/2$, since every such triple is counted twice.
	Hence, we have at most $\binom{p}{2}(n-p)$ $3$-term zero-sums. The maximum of $\binom{p}{2}(n-p)$ is taken at $p=n-\lfloor \frac{n}{3}\rfloor$, hence the claim.\qed
	
	\subsection{$r$-term sums, $r$ large}
	
	Here we prove Theorem \ref{nagyr}. In order to do this, we introduce a new (auxiliary) function.

	\begin{definition} Let $\s_r(A, \alpha)$ denote the number of $r$-term sums in $A$ which adds up to an element of $\QQ+\alpha$.\\
		Let
		$g(n, r)$ denote $$\max_{\alpha\in \mathbb{R}}\max_{A: A\in \Q^n \setminus (\QQ+\alpha/r)^n} \s_r(A, \alpha).$$ That is,  $g(n, r)$ maximizes the number of $r$-terms sums from the same equivalence class $\pmod \QQ$ which are not originated from a trivial construction $A\in (\QQ+\beta)^n $ for some $\beta\in \mathbb{R}$.
	\end{definition}
The next claim follows from the definition.
\begin{obs} $g(n,r)\ge \h(n,r)$.
\end{obs}

	\begin{lemma}\label{komp}
		$g(n,r)=g(n,n-r)$.
	\end{lemma}
	
	\begin{proof}
		Observe that $\s_r(A, \alpha)=\s_{n-r}(A, -\alpha+\sum_{a\in A}a)$.
	\end{proof}

	\begin{lemma}\label{lower_g} We have
		$$g(n,r)\geq \begin{cases}
			\ {n-1\choose r-1} \text{ \ if \ } r>\frac{n}{2},\\
			2{n-2\choose r-1}  \text{ \ if \ } r=\frac{n}{2}, \\ 
			\   {n-1\choose r} \text{ \ if \ } r<\frac{n}{2}.
		\end{cases}  $$ 
	
	Moreover, 	$$\h(n,r)\geq \begin{cases}
		\ {n-1\choose r-1} \text{ \ if \ } r>\frac{n}{2},\\
		2{n-2\choose r-1}  \text{ \ if \ } r=\frac{n}{2}.
	\end{cases}  $$ 
	\end{lemma}
	
	\begin{proof}
		Suppose that $r\neq \frac{n}{2}$. Then choose the elements of $A$ such that all but one belong to the same equivalence class
		$\QQ+\alpha$. Here, we either consider the $r$-sets of $A$ containing the exceptional element (case $r>\frac{n}{2}$) or we consider the $r$-sets of $A$ not containing the exceptional element (case $r<\frac{n}{2}$). The exceptional element may be from $\QQ-(r-1)\alpha$, which shows the bound on $\h(n,r)$ when $r>\frac{n}{2}$.\\
		If $r=\frac{n}{2}$, then let $A$ consists of two elements from the same equivalence class $\QQ+\beta$ and $n-2$  from another equivalence class $\QQ+\alpha$, $\beta\neq \alpha$. Then $\s_r(A, \alpha(r-1)+\beta)=  2{n-2\choose r-1}$. Setting $\beta=-(r-1)\alpha$ gives the bound on $h(n,n/2)$.
	\end{proof}
	
	\begin{lemma}\label{rec_g}
		We have either
		$g(n,r)\leq g(n-1,r)\cdot \frac{n}{n-r}$, or $g(n,r)={n-1\choose r-1}$, or $g(n,r)={n-1\choose r}$.
	\end{lemma}

	\begin{proof} We do a double counting argument.
		Let us take an $n$-element set $A$ with $\s_r(A, \alpha)=g(n,r)$ and consider the sum $$\sum_{A'\subset A \ : \ |A\setminus A'|=1} \s_r(A', \alpha).$$ 
		On the one hand, it is clearly bounded from above by $n\cdot g(n-1, r)$, provided that there is no $n-1$-element subset which falls into some equivalence class $\QQ+\beta$. In such cases, we indeed get $g(n,r)={n-1\choose r-1}$ or $g(n,r)={n-1\choose r}$, depending on the relation of $r$ and $n$.\\
		On the other hand, every $r$-term sum of $\s_r(A, \alpha)$ is counted exactly $n-r$ times in the summation. This completes the proof.
	\end{proof}
	
	\begin{lemma}\label{rec2_g} We have either
		$g(n,r)\leq g(n-1, r-1)+g(n-1,r)$ or $g(n,r)={n-1\choose r-1}$, or $g(n,r)={n-1\choose r}$.
	\end{lemma}
	
	\begin{proof} Once again, if the maximum of $g(n,r)$ is attained at a set $A$ described in Lemma \ref{lower_g} (case $n\neq 2r$), then we have either $g(n,r)={n-1\choose r-1}$, or $g(n,r)={n-1\choose r}$.
		If there is no equivalence class $\QQ+\alpha$ admitting $n-1$ elements from $A$, then let us chose an element $a\in A$. We have at most $ g(n-1, r-1)$ of those $r$-sets which contain $a$ and has the same sum modulo $\QQ$, and  $g(n-1,r)$ of those $r$-sets which do not contain $a$ and has the same sum modulo $\QQ$, hence the claim.
	\end{proof}
	
	\begin{theorem}\label{g(n,r)pontos}
		For every $r<\frac{n}{2}$, we have $g(n,r)=\binom{n-1}{r}$, while  $g(n,r)=2\binom{n-2}{r-1}$ holds for $r=n/2$.
	\end{theorem}
	
	We state the immediate consequence of this before the proof.
	
	\begin{cor}\label{kov}
		For every $n>r>\frac{n}{2}$, we have $g(n,r)=\binom{n-1}{r-1}$. Indeed, we can apply Lemma \ref{komp} to derive $g(n,r)=g(n,n-r)=\binom{n-1}{n-r}=\binom{n-1}{r-1}$ from Theorem \ref{g(n,r)pontos}. 
	\end{cor}
	
	\begin{proof}[Proof of Theorem \ref{g(n,r)pontos}] We prove by induction on $r$. The base case $r=1$  follows directly from the definition of the function $g$. Suppose now that the statement holds for all pairs $(n, k)$ for $k\leq r$ and we wish to derive the statement for $g(n,r+1)$.\\
		First, observe that we already know that $g(n,r+1)=\binom{n-1}{r}$ for $n\le 2r+1$ from  the combination of Corollary \ref{kov} and Lemma \ref{komp}. \\
		Next, building on the above observation, we apply Lemma \ref{rec_g} to determine 
		$g(2r+2,r+1)$ and get that 
		$$g(2r+2,r+1)\leq g(2r+1,r+1)\frac{2r+2}{r+1}=2\binom{2r}{r},$$ since the other two options yield smaller bounds equal to $\binom{2r+1}{r}$. However, we have a matching lower bound from Lemma \ref{lower_g}.\\
		We continue with the value of $g(2r+3,r+1)$. Suppose to the contrary that $\s_{r+1}(A, \alpha)=g(2r+3,r+1)> \binom{2r+2}{r+1}$, for some $2r+3$-element set $A$ and a real number $\alpha$. Such an $A$ cannot contain $n-1=2r+2$ elements from the same equivalence class, as it would only give $\binom{2r+2}{r+1}$ sums from the same class.  Take an arbitrary pair of elements $x, y\in A$. The number of $r+1$-terms sums which contribute to $\s_{r+1}(A, \alpha)=g(2r+3, r+1)$ and contain $x$ is at most $g(2r+2, r)$, which equals to  $\binom{2r+1}{r}$ by the inductional hypothesis.
		Thus we must have more than ${2r+2\choose r+1}-{2r+1\choose r}={2r+1\choose r+1}$ sums from $\QQ+\alpha$ which do not contain $x$. %The same condition holds for $y$, 
		Take two families of $r+1$-sets: the first family consists of those sets which have sum from $\QQ+\alpha$, and does not contain neither $x$ nor $y$. The second family  consists of the complement with respect to $A\setminus \{x\}$ of those sets which have sum from $\QQ+\alpha$, and does not contain $x$ but does contain $y$. 
		Each of these families consists of sets with equal set sums modulo $\QQ$. Note  that the elements of these sets are from the $2r+1$ elements of $A\setminus \{x,y\}$. However, since the cardinalities of the families add up to a number exceeding ${2r+1\choose r+1}$, hence there are sets belonging to both families and thus all these $r+1$-sets have equal set sums modulo $\QQ$. Consequently, all the sets in the first family and their complement w.r.t. $A\setminus \{x\}$  have set sum from $\QQ+\alpha$, moreover, there is a one-to-one correspondence between the elements of the first and second family thus   each has size at least $\frac{1}{2}{2r+1\choose r+1}$.
		\\ The same can be stated if we exchange the roles of $x$ and $y$. However, if a set $B\subseteq A\setminus \{x,y\}$ is counted in $\s_{r+1}(A, \alpha)$, and so is $A\setminus \{x\}\setminus B$ and $A\setminus \{y\}\setminus B$, it means that $x$ and $y$ are from the same equivalence class $\pmod \QQ$. But $x$ and $y$ were chosen arbitrarily, and all the elements of $A$ can not fall into the same equivalence class, so this leads to a contradiction.
		%; and the same holds for $y$.
		% \todo[inline]{még beírni}

		%$A-a+a=A>{2k+1\choose k+1}$, azonban $V\setminus \{x,y\}$-ban összesen ${2k+1\choose k+1}$ darab $k+1$-uniform összeg választható ki. Emiatt van az $x$-et nem tartalmazó összegek közül olyan, aminek a komplementere is $V\setminus \{x\}$-beli egy kiválasztott összeg. Mivel a komplementerek is egymással azonos osztályúak voltak, ezért ekkor minden $y$-t tartalmazó összegnek a komplementere is ebbe az azonos osztályba kell, hogy kerüljön.\\
		%Azaz ezek elemszámaira $A-a=a=\frac{A}{2}>\frac{1}{2} {2k+1\choose k+1}$, ennyi azonos osztályú összeg van $V\setminus \{x,y\}$-ban. Szimmetrikusan elmondhatjuk ezt $y$-ból is: ha ehhez $B$ azonos osztályú összeg tartozik, ami nem tartalmazza, amiből $B-b$ nem tartalmazza $x$-et se, akkor hasonlóan $B-b=b=\frac{B}{2}>\frac{1}{2} {2k+1\choose k+1}$.\\
		
		%Így $V\setminus \{x,y\}$-ban összesen találtunk $k+1$-uniform összegekből $B-b+A-a>{2k+1\choose k+1}$ darab azonos osztályút, ami több, mint az összes lehetőség, így van köztük megegyező. Ezeknek a $V\setminus \{x\}$, illetve $V\setminus \{y\}$-beli komplementerei is ki lesznek választva, amelyek megegyeznek, kivéve hogy az egyik $x$-et, a másik $y$-t tartalmazza. Mivel ezek a komplementer halmazból képzett összegek is egy osztályba esnek, ezért $x$ és $y$ is egy osztályba tartozik.\\
		%Beláttuk, hogy tetszőleges két szám egy osztályba kell, hogy tartozzon, ami ellentmondás, mivel nem lehet minden szám egy osztálybeli. Ebben az esetben valóban $g(2k+3,k+1)\leq {2k+2\choose k+1}$.\\ 

		Thus we get that $g(2r+3,r+1)\le \binom{2r+2}{r+1}$ holds. 
		
		\vspace{0.1cm}
		
		For $n>2r+3$, we can apply Lemma \ref{rec_g} to get that
		
		$$g(n,r+1)\leq g(n-1,r+1)\cdot \frac{n}{n-r-1}\leq {n-2\choose r+1}\cdot \frac{n}{n-r-1}\leq {n-2\choose r+1}\cdot \frac{n-1}{n-r-2}={n-1\choose r+1}$$ holds, otherwise we have  $g(n,r+1)={n-1\choose r}$ or $g(n,r)={n-1\choose r+1}$. Since $r+1<\frac{n}{2}$, we indeed get $g(n,r+1)\leq {n-1\choose r+1}$ and complete the inductional step by observing  that 
		$g(n,r+1)\leq {n-1\choose r+1}$ follows from Lemma \ref{lower_g} which proves the equality
		$g(n,r)={n-1\choose r}$ provided that $r<\frac{n}{2}$ holds.      
	\end{proof}
			
		\begin{proof}[Proof of Theorem \ref{nagyr}]
			Lemma \ref{lower_g} and Theorem \ref{g(n,r)pontos} together proves that $g(n,r)=\h(n,r)$ for $r\ge n/2$, and determines the exact value, in accordance with the claim.
		\end{proof}

	%\begin{constr}\label{mainCon}
	%\end{constr}
	\section{Irrational sum free sets, non-uniform case}
	
	In this section, we prove Theorem \ref{osszes} on the maximum number of rational sums.

	\begin{prop} Let  $\alpha$ be an irrational number, and choose $A=A_1\cup A_2$ such that $A_1\subset \QQ+\alpha, A_2\subset \QQ-\alpha$, with cardinalities  $|A_1|=\lfloor \frac{n}{2} \rfloor, |A_2|=\lceil \frac{n}{2} \rceil$. Then $\h(A)= {n\choose {\lfloor \frac{n}{2} \rfloor}}$.
	\end{prop}
	\begin{proof}
		Clearly, we get rational sums if and only if a subset $A'\subseteq A$ consists of the same number of elements from $A_1$ and $A_2$. Thus, we get 
		$$\cH(A)=\sum_{i=0}^{\lfloor \frac{n}{2} \rfloor} {{\lfloor \frac{n}{2} \rfloor}\choose i}\cdot {{\lceil \frac{n}{2} \rceil}\choose i} = \sum_{i=0}^{\lfloor \frac{n}{2} \rfloor} {{\lfloor \frac{n}{2} \rfloor}\choose {\lfloor \frac{n}{2} \rfloor -i}}\cdot {{\lceil \frac{n}{2} \rceil}\choose i} = {n\choose {\lfloor \frac{n}{2} \rfloor}}.$$
	\end{proof}
	
	Next, we prove that the value attained above is actually the maximum.
	
	\begin{prop}\label{32} For every set $A$ of $n$ irrational numbers, $\h(A)\le {n\choose {\lfloor \frac{n}{2} \rfloor}}$.
	\end{prop}
	
	We introduce a definition before the proof. 
	
	\begin{defi} Let $A\in \Q^n $ and $\phi$ be a core homomorphism w.r.t $A$ (c.f. Lemma \ref{key_homo}). The\textit{ signed size of a sum} of a subset $A'\subset A$ is $|\{a\in A': \phi(a)>0\}|-|\{a\in A': \phi(a)<0\}|$.
	\end{defi}
	
	\begin{proof}[Proof of Proposition \ref{32}]
		Let us define a partially ordered set (poset) on the subsets of $A$. \\ $A'\preccurlyeq A''$ if and only if $\phi(a)<0$ for all $a \in A'\setminus A''$ and $\phi(a)>0$ for all $a \in A''\setminus A'$. This will give a graded poset, as the signed sizes of the sums serve as $n+1$ levels. The maximal chains in the poset clearly consist of $n+1$ subsets, one from each level. The number of  maximal chains is $n!$. Indeed, arriving from the least element, namely $\{a:\phi(a)<0\},$ to the greatest element  $\{a:\phi(a)>0\}$ of the poset, we can take an arbitrary permutation of the elements of $A$ to either add one (which have positive $\phi$ value) or drop one (having negative $\phi$ value). Now observe that crucially, there is at most one subset $A'$ in every maximal chain which has a sum $\sum_{a\in A'}\phi(a)$ equals to zero, since the sum is strictly monotonically increasing along any chain.  On the other hand, suppose that $|\{a\in A:\phi(a)>0\}|=p$ and hence $|\{a\in A:\phi(a)<0\}|=n-p$. Then every set $A'\subset A$ which has signed size $t$ belongs to $(p-t)!(n-p+t)!$ maximal chains, and  $(p-t)!\cdot (n-p+t)!\geq {\lfloor \frac{n}{2} \rfloor}!\cdot {\lceil \frac{n}{2} \rceil}!$  holds.  Consequently, the number of sets $A'$ for which $\sum_{a\in A'}\phi(a)=0$ is not more than $\frac{n!}{{\lfloor \frac{n}{2} \rfloor}!\cdot {\lceil \frac{n}{2} \rceil}!}$, what we wanted to prove. 
	\end{proof}
	
	\section{Irrational sum free sets in intervals}
	
	In order to prove Theorem \ref{ordered}, we first show a matching lower bound via a construction.
	
	\begin{prop}\label{ord_also}  $\h_{ord}(n)\geq\lfloor {\frac{n^{2}}{4}} \rfloor.$
	\end{prop}
	Note that this is roughly half of the possible number of sums $\binom{n}{2}$.
	\begin{proof}
		Let $\alpha$ be an irrational number, and let us take an $n$ element set $A$ for which\\ $a_{i}\in \QQ+\alpha\cdot(-1)^i$ such that $\lfloor a_i\rfloor=i$  hold. Then it is obvious that the intervals $I_{k, \ell}:=[k, k+2\ell]$ contain the same number of elements from $\QQ-\alpha$ and $\QQ+\alpha$, provided that $0<k, \ell\in \Z$ and $k+2\ell\le n+1$.
		Thus, the number of rational sums corresponding to sets of form $A\cap I$ is at least $\sum_{\ell=1}^{\lfloor\frac{n}{2}\rfloor}(n-2\ell+1)=\lfloor {\frac{n^{2}}{4}} \rfloor.$ 
	\end{proof}
	
	Next, to complete the proof of Theorem \ref{ordered}, we show the corresponding upper bound. To this end, we define special subsets of every ordered $n$-set.
	
	\begin{defi}[Sum-chain] For a set of $n$ ordered elements $A=\{a_1, a_2, \ldots, a_n\}$, a {\em sum-chain} is a set of formal sums  $\cC_{s,t}:=\{\sum_{i=s}^k a_i \mid k\in \{s, s+1, \ldots t\}\}$ for some  pair $(s,t)$, with $1\le s\le t \le n$.\end{defi}
	
	\begin{prop}\label{ord_felso}  $\h_{ord}(n)\leq\lfloor {\frac{n^{2}}{4}} \rfloor.$
	\end{prop}

	\begin{proof} Take an arbitrary set $A$ of $n$ irrational numbers $A=\{a_1< a_2< \ldots <a_n\}$.
		First, observe that every sum-chain $\cC_{s,t}$ contains at most $\frac{t-s+1}{2}$ rational sums. Indeed, if an element $\sum_{i=s}^k a_i$ of a sum-chain is rational, then clearly the previous element $\sum_{i=s}^{k-1} a_i$ cannot be rational.  Moreover, the first formal sum $\sum_{i=s}^{s} a_i=a_s$ in $\cC_{s,t}$ is always irrational, and this implies the observation.\\ To complete the proof, note that every sum in view (of form $A\cap I$ for some interval) is contained in exactly one of the sums $\cC_{s,n}$ for some $s$. Hence, the total number of rational sums is bounded above by $\sum_{s=1}^n \lfloor\frac{n-s+1}{2} \rfloor= \lfloor\frac{n^2}{4} \rfloor$. 
	\end{proof}
	
	\section{Concluding remarks and open problems}
	
	We determined the maximum number of rational $r$-term sums for $n$-sets of irrational numbers when $r=3$ or $r\geq n/2$. It remains open to determine $\h(n,r)$ when $4\le r<n/2$, but the proof presented above suggest that the extremal set $A$ consists of elements from exactly two equivalence classes modulo $\QQ$.

	\begin{constr}\label{mainConst}
		Let $\alpha$ be an irrational number, and $1<r<n$. Let $A$ consists of $\lfloor\frac{n}{r} \rfloor$ elements from $\QQ-\alpha(r-1)$ and  $\left(n-\lfloor\frac{n}{r} \rfloor\right)$ elements from $\QQ+\alpha$.
	\end{constr}
	
	Construction \ref{mainConst} in turn proves the following
	\begin{prop}
		$\h(n,r)\geq \binom{n-\lfloor\frac{n}{r} \rfloor}{r-1}\cdot \lfloor\frac{n}{r} \rfloor. $
	\end{prop}

	\begin{conj}\label{mainConj}
		Suppose that $1<r<n$.    Then the maximum number of $r$-subsets with rational sums of an $n$-set of irrational numbers is attained in Construction \ref{mainConst}.
	\end{conj}
	
	Note that we proved Conjecture  \ref{mainConj} for $r\in \{2,3\}$ or $n/2\le r<n$.
	
	We pointed out  the connection between our main problem and  zero-sum problems, i.e., that we can reformulate our problem in terms of maximizing $k$-term zero-sums in $n$-element sequences or rationals, or equivalently, integers, via the application of suitable vector space homomorphisms.
	
	In general, our setting can be interpreted as having a vector space $V$ and a subspace $W$, and we wish to find many sums admitting values from $W$ while the elements themselves lie outside $W$. Taking a vector space homomorphism  in which $W$ is the kernel as we suggested gives rise to zero-sum problems.

\end{document}